\documentclass[%
 aip,%
 jmp,%
 amsmath,amssymb,%
 preprint,%
 author-year,%
]{revtex4-1}

\usepackage{graphicx}% Include figure files
\usepackage{dcolumn}% Align table columns on decimal point
\usepackage{bm}% bold math
\usepackage{amsthm,mathrsfs}
\usepackage{subfigure}

\newtheorem{lemma}{Lemma}
\newtheorem{remark}{Remark}
\newtheorem{corollary}{Corollary}
\newtheorem{assumption}{Assumption}
\newtheorem{theorem}{Theorem}

\usepackage{hyperref}
\hypersetup{
    colorlinks=true,
    linkcolor=blue,
    filecolor=magenta,      
    urlcolor=cyan,
    breaklinks=true
}

\newcommand{\rd}{\mathrm{d}}

\begin{document}

\title[]{Long-time asymptotics of non-degenerate non-linear diffusion equations}

\author{Ivan C. Christov}
\email{christov@purdue.edu}
\affiliation{School of Mechanical Engineering, Purdue University, West Lafayette, Indiana 47907, USA}

\author{Akif Ibraguimov}
\email{Akif.Ibraguimov@ttu.edu}
\affiliation{Department of Mathematics \& Statistics, Texas Tech University, Lubbock, Texas 79409, USA}

\author{Rahnuma Islam}
\email{rahnuma.islam@ttu.edu}
\affiliation{Department of Mathematics \& Statistics, Texas Tech University, Lubbock, Texas 79409, USA}
\date{\today}% It is always \today, today,

\begin{abstract}
We study the long-time asymptotics of prototypical non-linear diffusion equations. Specifically, we consider the case of a non-degenerate diffusivity function that is a (non-negative) polynomial of the dependent variable of the problem. We motivate these types of equations using Einstein's random walk paradigm, leading to a partial differential equation in non-divergence form. On the other hand, using conservation principles leads to a partial differential equation in divergence form. A transformation is derived to handle both cases. Then, a maximum principle (on both an unbounded and a bounded domain) is proved, in order to obtain bounds above and below for the time-evolution of the solutions to the non-linear diffusion problem. Specifically, these bounds are based on the fundamental solution of the linear problem (the so-called Aronson's Green function). Having thus sandwiched the long-time asymptotics of solutions to the non-linear problems between two fundamental solutions of the linear problem, we prove that, unlike the case of degenerate diffusion, a non-degenerate diffusion equation's solution converges onto the linear diffusion solution at long times. Select numerical examples support the mathematical theorems and illustrate the convergence process. Our  results have implications on how to interpret asymptotic scalings of potentially anomalous diffusion processes (such as in the flow of particulate materials) that have been discussed in the applied physics literature.
\end{abstract}


\maketitle

\section{Introduction}
\label{sec:intro}

Almost two centuries ago, Robert Brown observed the apparently random motion of pollen particles on the surface of a liquid layer \citep{Brown28}. Since then, what has become known as ``Brownian motion'' \citep{Frey04} continues to offer scientific insights into microscopic phenomena, including in frontier areas such as microrheological measurements of complex fluids \citep{Zia18}. However, Brown's work did not yield a working theory. Three-quarters of a century later, Einstein proposed the first complete mathematical description of this phenomenon \citep{Einstein05}. Specifically, Einstein showed that the spread of the pollen particles obeys a \emph{diffusion} process, when viewed macroscopically in the sense of a probability distribution of where the particles might be found. The diffusion process arises from the random motion (walk) of the pollen particles caused by their endless collisions with the thermally agitated molecules of the fluid in which they are suspended. The final piece in solving the puzzle of Brown's experiments was Einstein's determination (contemporaneously with \citet{Sutherland05} and \citet{Smoluchowski06}) of the diffusivity in terms of the fluid properties (viscosity, temperature) via the previous result of Stokes on viscous fluid drag \citep{Stokes51}.

In this classical example of diffusion, the process is governed by a \emph{linear} equation, specifically a linear evolutionary parabolic partial differential equation (PDE) for the probability, concentration or another related quantity that describes the collection of particles at the macroscale. The diffusivity is constant and set by the Stokes--Einstein formula. Since this classical work, however, diffusion equations have been derived (and analyzed) as the governing equation of various other phenomena as well, ranging from flows through porous media \citep{Barenblatt52,Philip70} to the motion of free interfaces bounding thin liquid films \citep{Oron97} to high-temperature shock wave phenomena \citep{Zel67}. In contrast to Brownian motion and Einstein's theory, the latter examples lead to \emph{non-linear} diffusion problems in which the diffusivity is a function of the dependent variable (say, concentration or probability).  Interestingly, almost all of these examples feature \emph{degenerate} non-linear diffusion, i.e., the diffusivity vanishes when the dependent variable (or its gradient \citep{Celik17}) vanishes \citep{Vazquez07}. 

A more recent example of a diffusion process concerns the flows of granular materials, such as sand. These materials are macroscopic and the thermal fluctuations of ``molecules'' are irrelevant. The granular material must be driven by external forces to flow, which gives rise to collective diffusion. Specifically, during flow, particles (grains) collide with each other. Although these collisions are deterministic, they occurs so often and in such variety that one may consider the end result to be random particle velocity fluctuations, much like to those imparted on pollen by the fluid's molecules in Brown's 1828 experiments. For a collection of identical (in shape, size and density) particles, one might expect that the diffusivity is independent of concentration \citep{Lacey54,Savage93}, and indeed experiments support this claim \citep{Cahn66,Zik91}. %Cambpell97,Utter04
However, granular materials are often highly heterogeneous, i.e., they are mixtures of particles of different shapes, size and density  \citep{Ottino00,Umbanhowar19}. In the case of particles of different sizes, depending on the mixture proportions, the positive time interval between small-small or small-large binary collisions depends on the number of particles of a given type locally. The small-small and small-large collisions are not identical as the larger particles impact a bigger force onto the smaller ones, therefore produce a different free jump. Nonetheless, the free jump frequency distribution itself could be assumed independent of the particle concentration because the jumps are entirely set by the collision physics (mass, velocity, coefficient of restitution, etc.). 

Then, there are three main parameters that are involved in Einstein's random walk paradigm \citep{Einstein05}: the frequency distribution of free jumps $\varphi(\delta)$, the length of a free jump $\delta$, and time interval $\tau$ within which particles perform a free jump. A thought experiment along the lines of Einstein can incorporate non-linearity in the random walk model. The influence of particle concentration can be taken into account through the interval of free jumps. In \S\ref{Einstein principle}, we will discuss how the diffusivity will depend, through $\tau$, on how many large (or small particles) are in the vicinity of a spatial location. This observation leads to concentration-dependent diffusion of poly-disperse granular materials \citep{Fischer09}. Interestingly, however, the non-constant diffusivity in these cases is \emph{not} degenerate \citep{Ristow99,Dury99}, as even in the absence of large particles nearby, the small particles still collide and, thus, ``diffuse'' \citep{Christov12}.
% and simulations \cite{tsm11} have suggested that the diffusivity depends on the local concentration 

Thus, we have provided ample motivation that in many applications involving the flow of liquids, gases, and even particulates, the dependent variable in the problem, such as density, pressure or concentration, can be governed by a non-linear evolutionary parabolic PDE, i.e., a diffusion equation. In the case of degenerate diffusivity, a large mathematical literature exists discussing the qualitative properties of solutions, asymptotics, and so on \citep{DiBenedetto93,Vazquez07}. On the other hand, the case of non-degenerate diffusion has not received as much attention. Nevertheless, in applications involving such PDE, one observes ``regular'' diffusive scalings at long times (in the sense of \emph{intermediate asymptotics} \citep{Barenblatt96}), which has caused no small amount of controversy in interpreting experiment and simulation data \citep{Christov12}.\footnote{Here, by ``scalings'' we mean that the solution $u(x,t)$ can, at any $t$ sufficiently large, be transformed/collapsed as $u(x,t) \mapsto  \kappa_1 t^{\mathfrak{n}_1}U\big(\kappa_2 xt^{-\mathfrak{n}_2}\big)$, to a good approximation, into a universal profile $U(\cdot)$, for some suitable dimensional constants $\kappa_{1,2}$ and \emph{scaling exponents} $\mathfrak{n}_{1,2}$.}

Motivated by the need for a clear mathematical answer regarding the long-time behavior of solutions (distinct from the short-time behavior studied by \citet{Christov12,Sekimoto19}) of non-degenerate non-linear parabolic equations, in the present work, we establish new mathematical results showing that the long-time asymptotics of solutions are given by a Gaussian profile and its corresponding ``normal'' (as apposed to ``anomalous'') scalings obtainable from the linear diffusion equation. Specifically, using the maximum principle (\S\S\ref{sec:max_bd} and \ref{sec:max_ubd}), we obtain estimates, from above and below, on the solution of the non-linear problem (applicable to equations in both divergence and non-divergence form, as shown in \S\ref{mapping}), proving that the solution to the non-linear problem converges to the Green's function of the linear problem (suitably shifted). These mathematical results are illustrated via numerical simulations in \S\ref{sec:numerical}, and conclusions are stated in \S\ref{sec:conclusion}.

%%%%%%%%%%%%%%%%%%%%%%%%%%%%%%%%%%%%%%%%%

\section{Position of the problem}
As a generic model for all the previously mentioned phenomena, consider:
\begin{equation}\label{p-u-equation} 
\Tilde{L}u=\frac{\partial u}{\partial t} - \frac{\partial}{\partial x}\left( P(u)\frac{\partial u}{\partial x}\right)=0,\qquad (x,t)\in (-\infty,+\infty)\times(0,\infty),
\end{equation}
where $P(u)$ is a polynomial with $P(u) > 0$ $\forall u\ge 0$, and $u$ is a scalar quantity that is characteristic of the physical systems (to be made precise below). Equation~\eqref{p-u-equation} is subject to a non-negative localized initial condition:
\begin{equation}
u(x,0) = u_0(x).
\label{eq:u_ic}
\end{equation}
To fully pose the problem, appropriate growth conditions should be satisfied at $|x|=\infty$ by the solution and the initial condition:
\begin{equation}
\lim_{|x|\to\infty} u(x,t) |x|^{\gamma} = 0 \qquad \forall t\ge 0,
\label{eq:decay_bc}
\end{equation}
for some $\gamma>0$. The solution, which is initially non-negative, should remain non-negative for all times:
\begin{equation}
u(x,t) \ge 0, \qquad \forall (x,t)\in (-\infty,+\infty)\times[0,\infty).
\end{equation}

Observe that solutions to Eqs.~\eqref{p-u-equation}--\eqref{eq:decay_bc} obey a conservation principle (see also \S\ref{conserv-law}): $\int_{-\infty}^{+\infty} u(x,t)\, \rd x = \int_{-\infty}^{+\infty} u_0(x)\, \rd x > 0$ $\forall t\ge0$, as easily shown by direct integration and application of the decay condition as $|x| \to \infty$. 

Since we are studying a generic mathematical problem, we do not concern ourselves with the units. Physically, this just means that we have made $x$ dimensionless by some characteristic domain length $x_c$, we have made $P$ dimensionless by some characteristic diffusivity $P_c$ (such that, for $P\mapsto P/P_c$, we may take $P(0)=1$ now), we have made $u$ dimensionless by some characteristic scale $u_c$ (say, such that for $u\mapsto u/u_c$, $\int_{-\infty}^{+\infty} u(x,t)\, \rd x = 1$ $\forall t\ge0$), and we have made $t\mapsto t/t_c$ dimensionless by the characteristic diffusion time $t_c = x_c^2/P_c$. Here, the subscript `$c$' stands for `characteristic.'

In Eq.~\eqref{p-u-equation}, we consider the case in which $P(u)$ is a polynomial such that $P(u) > 0$ $\forall u\ge 0$. Additionally, we will assume $P(u)$ has no real roots. Hence, Eq.~\eqref{p-u-equation} is non-degenerate, which is specifically the case of interest here. A particular example we are interested in, based on previous studies \citep[see, e.g.,][]{Ristow99,Christov12}, is $P(u) = a_0 + a_1 u$ with $a_0>0$, which is of the form assumed above. Equation \eqref{p-u-equation} is usually derived from conservation principles (see \S\ref{conserv-law}) in conjunction with the physical law that the flux is proportional to the gradient of the scalar function $u$. Examples of the latter law are Darcy's, Fourier's, Fick's, etc.\  \citep{BSL}.
%, and they are justified by non-physical experiment. 

At the same time, this process can be also be modeled using the Einstein paradigm (\S\ref{Einstein principle}), which he used to derive the linear diffusion equation describing Brownian motion macroscopically. Therefore, we next discuss this derivation, in detail, for the non-linear case. We will show that the governing PDE has non-divergence form and can be derived from a probabilistic model via a thought experiment. To connect the non-divergence-form equation (obtained from Einstein's paradigm) to the divergence-form equation~\eqref{p-u-equation} (obtained from a conservation principle), in \S\ref{mapping}, we show there exists a closed-form mapping between the solutions of the two.

\section{Einstein paradigm: ``From random motion of particles to diffusion"}
\label{Einstein principle}

In this section, we explain how non-linearity can be incorporated into Einstein's random walk model of Brownian motion. In the celebrated work of \citet{Einstein05,Einstein56}, a mathematical model is derived on the basis of a thought experiment (or, \emph{Gedankenexperiment} in Einstein's own terminology \citep{Gedank}). On the basis of its generality, this model can then be applied to a number of physical processes, in which a random walk occurs \citep[see, e.g.,][\S1.2.1]{G09}, arising in physics, chemistry and engineering. Here, for completeness and clarity, we  summarize Einstein's derivation, exactly as it appears in his original work. In doing so, we highlight the points of departure of the non-linear model considered in the subsequent sections of the present work. 

Four axioms are used to complete the derivation, which we formulate as assumptions:

\begin{assumption}
There exists a time interval $\tau$, which is very small compared to the observable time intervals but large enough that the motions, performed by particles during two consecutive time intervals $\tau$, can be considered as mutually independent events.
\end{assumption}

\begin{assumption}
The distance traveled during the time interval $\tau$, without undergoing a collision, is called the ``free jump" and has a finite size, $\delta$. 
\end{assumption}
    
\begin{assumption}
The particles are not allowed to interact chemically (i.e., they cannot agglomerate, or breakup, or react with a solvent).  
\end{assumption}

Let the total number of particles present in the system be $N$. The number $\rd N$ of particles experiencing a displacement that lies between $\delta$ and $\delta +\rd \delta$ in the time interval $\tau$ is given by
\begin{equation}\label{dn}
     \rd N = N \varphi(\delta) \,\rd \delta,
\end{equation}
where $\varphi$ is the probability density function of particle jumps such that 
\begin{equation}\label{phi_normalization}
 \int_{\delta_\mathrm{min}}^{\delta_\mathrm{max}}\varphi(\delta) \,\rd\delta = 1.
\end{equation}
Einstein assumed that $\varphi$ is localized, i.e., it differs from zero only in a range of $\delta$ values about $\delta =0$. This assumption seems natural, but it does not have always be true for all physical diffusion processes.

\begin{remark}\label{einsteins-parad}
Note that in the Einstein paradigm, $\tau$,  $\delta$ and $\varphi$ are characteristics of the physical process. In general, these three parameters can be functions of both the spatial variable $x$ and the time variable $t$, as well as other physical quantities (depending on the problem). Furthermore, at this point, nothing prevents $\tau$, $\delta$ and $\varphi$ from also being functions of the dependent variable (and its derivatives), such as the number of particles $N$.

In the present work, however, we will assume that the length of the free jumps  $\delta$ and their frequency distribution $\varphi(\delta)$ are fixed (by the underlying physics) w.r.t.\ $N$. Then, the only parameter involved that can depend on $N$ is the time interval $\tau$, which leads to the non-linear nature of the diffusion process below.
\end{remark}

\begin{assumption}\label{conservation law}
Let $f(x, t)$ be the number of particles per unit volume. Then, the number of particles found at time $t+ \tau$ between two planes perpendicular to the $x$-axis, with abscissas $x$ and $x+ \rd x$, is given by
\begin{equation}\label{Einstein_conserv_eq}
f(x, t+\tau) \cdot \rd x =   \left(\int_{\delta_\mathrm{min}}^{\delta_\mathrm{max}} f(x+ \delta, t) \varphi(\delta) \rd  \delta \right)\cdot \rd x .
\end{equation}
\end{assumption} 
Next, by Caratheodory's theorem, there exists a function $\psi(x,t)$ such that 
\begin{equation}\label{Caratheodory}
f(x, t+ \tau)= f(x, t)+ \tau \psi(x,t+\tau) ,
\end{equation}
where
\begin{equation}
  \lim_{\tau\to 0} \psi(x,t+\tau)= \frac{\partial f(x,t)}{\partial t}.
\end{equation}
However, we shall not formally take the limit. Instead, since $\tau \ll t$ (see Remark~\ref{tau-finite} below), we make the approximation
\begin{equation}
    \psi(x,t+\tau)\approx \frac{\partial f(x,t)}{\partial t}.
\end{equation}

Nest, using the Taylor  expansion of $f(x+\delta, t)$ in powers of $\delta$, we obtain
\begin{equation}\label{taylor1}
 f(x+\delta, t)= f(x, t)+ \delta \frac{\partial f}{\partial x} + \frac{\delta^2}{2!} \frac{\partial^{2} f}{\partial x^{2}} + \cdots  .
\end{equation}
Thus, we obtain from Eqs.~\eqref{Einstein_conserv_eq} and \eqref{taylor1}:
\begin{equation}
    f(x, t)+ \tau \frac{\partial f}{\partial t}=  f(x, t)\underbrace{\int_{\delta_{min}}^{\delta_{max}} \varphi(\delta) \,\rd \delta}_{=1\text{ by Eq.~\eqref{phi_normalization}}} + \frac{\partial f}{\partial x} \int_{\delta_\mathrm{min}}^{\delta_\mathrm{max}} \delta \varphi(\delta) \,\rd \delta + \frac{\partial^{2} f}{\partial x^{2}} \int_{\delta_\mathrm{min}}^{\delta_\mathrm{max}} \frac{\delta^2}{2!} \varphi(\delta) \,\rd \delta + \cdots ,
\end{equation}
or
\begin{equation}\label{taylor3}
     \tau \frac{\partial f}{\partial t}= \int_{\delta_\mathrm{min}}^{\delta_\mathrm{max}} \delta \varphi(\delta) \,\rd \delta \cdot\frac{\partial f}{\partial x}   + \int_{\delta_\mathrm{min}}^{\delta_\mathrm{max}} \frac{\delta^2}{2!} \varphi(\delta) \,\rd \delta\cdot \frac{\partial^{2} f}{\partial x^{2}}  + \cdots.
\end{equation}

Now, we impose restrictions on the jump size distribution $\varphi$, to make precise Einstein's assumption on the localized nature of $\varphi$:

\begin{assumption}\label{delta^2>delta^4}
$|\delta_\mathrm{min}|\gg|\delta_\mathrm{min}|^2$ and $|\delta_\mathrm{max}|\gg|\delta_\mathrm{max}|^2$.
\end{assumption}

\begin{assumption} \label{Symmetry in movements} 
$\varphi(\delta) = \varphi(-\delta)$.
\end{assumption}

Due to Assumption \ref{delta^2>delta^4}, we keep only second-order (in $\delta$) terms in the right-hand side of Eq.~\eqref{taylor3}. Due to Assumption \ref{Symmetry in movements}, the first and further  odd moments vanish.  Now, we define the \emph{diffusivity} (diffusion coefficient) $D$ via the second moment of $\varphi$:
\begin{equation}\label{D-def}
    \frac{1}{\tau} \int_{\delta_\mathrm{min}}^{\delta_\mathrm{max}} \frac{\delta^2}{2!} \varphi(\delta) \,\rd \delta = D,
\end{equation}
whence Eq.~\eqref{taylor3} leads to the well-known linear diffusion equation for the function $f$ counting the number of particles per unit volume:
\begin{equation} \label{eq:1}
\frac{\partial f}{\partial t}= D  \frac{\partial^{2} f}{\partial x^{2}}.
\end{equation}

\begin{remark}\label{non-linarity}
According to our interpretation of the Einstein paradigm (Remark \ref{einsteins-parad}), the diffusion coefficient $D$ in Eq.~\eqref{D-def} can be a function of $x$ and $t$, or even $f$, via the time interval $\tau$. In other words, unlike previous works that apply the Einstein paradigm to non-linear diffusion \citep{Boon07,Lenzi19}, the non-linearity in the present context comes into play via $\tau$ and its possible direct dependence on $f$ specifically.
\end{remark}

\begin{remark}\label{tau-finite}
``Einstein's derivation is really based on a discrete time assumption, that impacts happen only at times 0, $\tau$, 2$\tau$, 3$\tau$, ...'' \citep[p.~5]{G09}. In other words, in the Einstein paradigm (Remark \ref{einsteins-parad}), $\tau$ is considered to be the \emph{finite} time between particle collisions. This microscopic time scale is assumed to be small compared to the observational (macroscopic) time scale, $\tau \ll t$, but it is \emph{not} taken to zero. Therefore, here, Eq.~\eqref{D-def} is interpreted as stated, not as a limiting process.
\end{remark}

\section{Non-linear parabolic equations in divergence and non-divergence form}\label{mapping}

\subsection{Non-linear model arising from the Einstein paradigm (equation in non-divergence form)}

First, we establish that the Einstein paradigm can be used to obtain a non-linear parabolic diffusion equation for a Brownian-like process.

\begin{assumption}\label{D(v)}
Let the number of particles per volume, $f(x,t)$, be linearly proportional to a scalar function, such as the concentration $v(x,t)$ in the medium, which is assumed to be homogeneous and isotropic. We  postulate that time-interval of free jumps $\tau \nearrow a_0^{-1}>0$ as $v\searrow0$. In other words, we consider the concentration-dependent diffusivity function 
\begin{equation}\label{D(v)-def}
D(v)=a_0+F(v),
\end{equation}
where $F(v)$ is a non-decreasing homogeneous function.
\end{assumption}

Under Assumption~\ref{D(v)} and for general for $x\in\mathbb{R}^d$, Eq.~\eqref{eq:1} takes the form 
\begin{equation}\label{non_diveregent_eq}
    Lv = \frac{\partial v}{\partial t}- D(v) \Delta v =0,
\end{equation}
where $D(v)$ is the diffusion coefficient for concentration $v$ at the point $(x,t)$, and $\Delta v = \sum_{i=1}^d\frac{\partial^2v}{\partial x_i^2} $ is the Laplacian operator applied to $v$. Taking the scalar function $v\geq 0$ to be non-negative, it follows from Assumption~\ref{D(v)} that Eq.~\eqref{non_diveregent_eq} is \emph{non-degenerate}.

\subsection{Non-linear model arising from the conservation law principle (equation in divergence form)}\label{conserv-law}

Another way to take into account of non-linearity is to employ
the more traditional conservation law for the density $\rho(x,t)$ of a substance of interest with attendant flux vector $\vec{J}(x,t)$ (i.e., the conservation of mass or ``continuity'' equation \citep{BSL,Dafermos16}). This equation, for a finite control volume, is written in divergence form as
\begin{equation}\label{conserv-dif-eq}
 \frac{\partial }{\partial t}\int_\mathcal{V}\rho \,\rd x + \oint_{\partial \mathcal{V}} \Vec{J}\cdot\vec{\nu} \,\rd s = 0 .
\end{equation} 
Here, $\mathcal{V}\subset \mathbb{R}^d$ is the control volume, $\partial \mathcal{V}\subset \mathbb{R}^{d-1}$ is its boundary (with unit normal vector $\vec{\nu}$), and $t\in\mathbb{R}$ is the time variable.

\begin{assumption}\label{therm}  
To allow for the consideration of different physical phenomena governed by the same equations, let the density  $\rho=A u+B$ be a linear function of some scalar $u(x,t)$, where $A>0$ and $B\geq 0$ are suitable dimensional constants. For example, $u$ can be the concentration.
\end{assumption}

\begin{assumption}\label{Fick}
Suppose that Fick's law \citep{Fick1855,Fick1855a} holds a.e. That is, the flux vector $\vec{J}$ is proportional to the gradient of concentration:
\begin{equation}\label{Fick-eq} 
    \vec{J}=-P \nabla \rho=-P A\nabla u,
\end{equation} 
where the proportionality factor $P$ is precisely the diffusivity (in general, a tensor of second rank) \citep{BSL}.
\end{assumption}

\begin{assumption}\label{incr-assump-cond}
In an isotropic medium, the diffusivity $P$ in Eq.~\eqref{Fick-eq} is a scalar. Specifically, let $P(\cdot)$ be a  non-decreasing function of $u$: $P(u)=a_0+G(u)$, where $G(u)\geq 0$.
\end{assumption}

\begin{remark}
Assumption~\ref{incr-assump-cond} is simply a restatement of Assumption~\ref{D(v)} but for the case of a divergence-form equation derived from the conservation law principle.
\end{remark}

Under the above assumptions, and applying Green's theorem for the arbitrary control volume $\mathcal{V}$, Eq.~\eqref{conserv-dif-eq} can be transformed to an equation in divergence form at space-time point $(x,t)$:
\begin{equation}\label{conserv-dif_eq}
\Tilde{L}u = \frac{\partial u}{\partial t} - \nabla\cdot\left(P(u) \nabla u\right) = 0 .
\end{equation}
Here, due to Assumption~\ref{incr-assump-cond}, the diffusivity is a non-decreasing function w.r.t.\ $u$, and we take the constant $a_0>0$ to be strictly positive. Then, for a scalar function $u\geq 0$ that is non-negative, $P(u)>0$ $\forall u$, and it follows that Eq.~\eqref{conserv-dif_eq} is \emph{non-degenerate}.

%%%%%%%%%%%%%%%%%%%%%%%%%%%%%%%%%%%%%%%%%%%%%%%%

\subsection{Mapping between solutions of equation in divergence and non-divergence form}

The two governing diffusion equations, i.e., Eqs.~\eqref{conserv-dif_eq} and \eqref{non_diveregent_eq},
introduced above are obviously related. Now, we prove that,  for any positive polynomial function $P(u)$ with non-negative coefficients satisfying Assumption~\ref{incr-assump-cond}, there exists a function $D(v)$ satisfying Assumption \ref{D(v)} s.t.\  $P(u)=D(v)$.

\begin{theorem}\label{u-to-v-thm}
Let $P(u)$ be a polynomial 
\begin{equation}
    P(u)=a_0+a_1u+\cdots+a_nu^n
    \label{eq:P(u)}
\end{equation}
with all non-negative coefficients and at least one strictly positive coefficient.
Consider the transition formula 
\begin{equation}\label{u-to-v}
    v=\int_0^u P(\xi) \,\rd\xi.
\end{equation}
Then, from monotonicity of integral it follows that there exists a function $F(v)$ s.t.
\begin{equation}\label{P-D}
P(u)=D(v),
\end{equation}
where $D(v)=a_0 + F(v)$ as given in Eq.~\eqref{D(v)-def}.
\end{theorem}

\begin{proof}
For first-order polynomials, the construction of the function $F(v)$ is explicit. Indeed for $n=1$, $P(u)=a_0+a_1 u$.
Let
\begin{equation}\label{Rahnuma}
    F(v)=-a_0+\sqrt{a_0^{2}+2 a_1 v}.
\end{equation}
Then, a direct substitution shows that:
$$D(v)=P(u).$$
\end{proof}

For the general case, this construction is not explicit. 

Next, we show that if $u$ and $v$ are related by the transition formula~\eqref{u-to-v}, and $D(v)$ satisfies Eq.~\eqref{P-D}, then the left-hand side of Eq.~(\ref{non_diveregent_eq}) takes the form:
\begin{equation}\label{non_diveregent_eq2}
    Lv: \frac{\partial v}{\partial t}- P(u) \Delta v .
\end{equation}
Observe that both functions $u$ and $v$ are involved in Eq.~\eqref{non_diveregent_eq2}.
%Note that the relation between $u$ and $v$ is such that $P(u)=D(v)$.

Denote by $\mathcal{C}^{2,1}$ the class of function that have continuous second derivatives w.r.t.\ $x$ and continuous first derivatives w.r.t.\ $t$. Then, we prove:
\begin{lemma} 
If $u, v \in \mathcal{C}^{2,1}$, then the transition formula \eqref{u-to-v} implies that
\begin{equation}\label{v-u-rel}
Lv = \frac{\partial v}{\partial t} - D(v) \Delta v = \frac{\partial v}{\partial t}  - P(u) \Delta v = P(u) \Tilde{L}u ,
\end{equation}
where $\Tilde{L}u$ is defined in \eqref{conserv-dif_eq}.
\end{lemma}

\begin{proof} 
Let $v= \psi (u)$, where $\psi (u)= \int_{0}^{u} P(\xi) \,\rd\xi$ and so, $P(u)= \psi'(u)$. Then, we compute:
\begin{equation*}
    \begin{split}
    Lv = \frac{\partial v}{\partial t} - P(u) \Delta v & = \psi '(u) \frac{\partial u}{\partial t} - P(u) \psi^{''}(u) |\nabla u|^2 - P(u) \psi'(u) \Delta u \\
    & = \psi'(u) \bigg[ \frac{\partial u}{\partial t}  - P(u)\frac{\psi^{''}(u)}{\psi'(u)}|\nabla u|^2- P(u) \Delta u \bigg]\\
    &=\psi'(u) \Tilde{L}u .
    \end{split}
\end{equation*}%
\end{proof}

\begin{corollary}
Since $\psi'(u)>0$, if $u$ is a solution of the equation $\Tilde{L}u=0$, then $v$ is solution of the equation $Lv=0$.
\end{corollary}

%%%%%%%%%%%%%%%%%%%%%%%%%%%%%%%%%%%%%%%%%%

\section{Maximum principle on a bounded domain}
\label{sec:max_bd}

In this section, we prove a maximum principle for the solution, following \citet{Ilyin02} \citep[see also][]{Landis}. Let $\mathbb{R}^{d}$ be the $d$-dimensional real Euclidean space. We also consider the $(d+1)$-dimensional space $\mathbb{R}^{d+1}$, in which the spatial coordinates are augmented by time: $(x,t)= (x_1, x_2,\hdots,x_d,t)$. 
Now, suppose that $U \subset \mathbb{R}^{d}$ is bounded, and $t>0$. We define the cylindrical region  $\Omega= U \times (0,T] \subset \mathbb{R}^{d+1}$, and its parabolic boundary $\Gamma = (U \times \{t=0\})\cup (\partial U \times (0,T])$. We also consider the layer $H\subset \mathbb{R}^{d+1}\cap(0,T]$. 

Again, let $\mathcal{C}^{2,1}(\Omega)$ be the class of continuous functions in $\bar{\Omega}$  that have two continuous derivatives in $x$ and one continuous derivative in $t$ inside the domain $\Omega$. Henceforth, both functions $u(x,t)$ and $v(x,t)$ are assumed to be in $\mathcal{C}^{2,1}(\Omega)$.

\begin{lemma}\label{maximum-principle_weak} 
For a given function $u \in \mathcal{C}^{2,1}$, if $P(u) \ge a_0 > 0$ and $Lv > 0$ in $\Omega$, then $v(x,t) \geq \min_{\Gamma(\Omega)} v  $ in $\Omega$.
\end{lemma}

\begin{proof} 
Suppose there is a point $(x_0, t_0) \in \Omega$ such that $\min_{\bar{\Omega}}v=v(x_0, t_0)< \min_{\Gamma(\Omega)} v$, then the minimum of $v(x, t)$ attained at $(x_0, t_0) \in \Omega$, for $x_0 \in U$ and $t_0 \leq T$.
Therefore, at the point $(x_0, t_0)$, $\nabla v= 0$, $\frac{\partial v}{\partial t} \leq 0$ and $\Delta v \geq 0$.

Consequently, $Lv \leq 0$ at the point $(x_0, t_0)$, which is a contradiction. 
\end{proof}

\begin{lemma}\label{maximum_principle_strong}  
For a given function $u \in \mathcal{C}^{2,1}$, if $P(u) \ge a_0 > 0$ $\forall u$ and $Lv \geq 0$ in $\Omega$, then $v \geq \min_{\Gamma(\Omega)} v$ in $\Omega$.
\end{lemma}
\begin{proof}
Consider the function
$w(x, t)= Kt+ v(x, t)$ for $(x, t) \in \Omega$ and  $t<T$, $K >0$. 
Then, $Lw= \frac{\partial w}{\partial t} - P(u) \Delta w= K+ \frac{\partial v}{\partial t} - P(u) \Delta v >0$.

By Lemma~\ref{maximum-principle_weak} and since $Lw 
> 0$, we obtain $w(x, t) \geq \min_{\Gamma(\Omega)} w$ in $\Omega$ for $K>0$. 
Thus, $K t + v(x,t) \geq \min_{\Gamma(\Omega)}v$ in $\Omega$, which implies $K T + v(x,t) \geq \min_{\Gamma(\Omega)}v$ in $\Omega$ for $K$. Taking $K\to0$, $v\geq \min_{\Gamma(\Omega)} v$.
\end{proof}

From standard  maximum principle follows the comparison lemma:
\begin{lemma}\label{comparison_theorem}
If $P(u) \ge a_0 > 0$ $\forall u$  (Assumption~\ref{incr-assump-cond}) and $Lv_1=\frac{\partial v_1}{\partial t}-P(u) \Delta v_1 \geq Lv_2=\frac{\partial v_2}{\partial t}-P(u) \Delta v_2$ in $\Omega$, and $v_1\geq v_2$ on $\Gamma$, then $v_1 \geq v_2$ in $\Omega$.
\end{lemma}

\begin{proof}
It is sufficient to apply the maximum principle to the function $w(x,t) = v_1 (x,t)-v_2 (x,t)$ using the properties of the functions $v_1$ and $v_2$.
\end{proof}

\section{Maximum principle on an unbounded domain}
\label{sec:max_ubd}

Let $r=\big(\sum_{i=1}^{d}x_{i}^{2}\big)^{1/2}$. As in the previous section, we follow  \citet{Ilyin02} and take into account that the coefficients of the operator $L$ are given by $P(u)$.

\begin{lemma}
 Suppose the function $u(x, t)$ is continuous in $\Omega=\mathbb{R}^d\times [0,T]$ such that $ 0<a_0\le P(u) \leq \epsilon(r) r^2+ C$ with $\lim_{r \to \infty}\epsilon(r)=0$ and $v(x,t)>-m$, $m>0$. Under these assumptions, if $Lv\geq 0$ and  $v|_{t=0}>0$, then $v(x,t)\geq 0$ for all $(x,t)\in\Omega $.

\end{lemma}

\begin{proof}
Consider, the auxiliary function $w(x,t)$, s.t.
\begin{equation*}
    w(x,t)=\frac{m}{r_0 ^2}\left(r^2+Kt\right)e^{\alpha t}+ v(x,t) ,
\end{equation*}
on the auxiliary domain consisting of the cylinder $Q_{r_0}= \{(r,t) \,|\, r \leq r_0, 0\leq t \leq T\}$, where the constants $K>0$ and $\alpha >0$ $\forall r_0>0$.
Indeed,
\begin{equation*}
    Lw = \frac{m}{r_{0}^{2}}e^{\alpha t} \big(K+ \alpha r^2 +K \alpha t - 2d P(u)\big)+Lv .
\end{equation*}

Since $Lv \geq 0$ and $P(u) = \epsilon(r) r^2+ C$, then for $r \geq 1$, we 
\begin{equation*}
    Lw \geq \frac{m}{r_{0}^{2}}e^{\alpha t} \left(K+ K \alpha t +\alpha r^2 - 2d \cdot \mathrm{o}(r^2+1)\right)>0 .
\end{equation*}
On the other hand, for $r<1$, $P(u)\leq M$ for some $M$. It follows that, if $\alpha> \frac{2d M}{r^2}$, then
\begin{equation*}
    Lw \geq \frac{m}{r_{0}^{2}}e^{\alpha t} (K + K \alpha t +\alpha r^2 - 2d M\big)>0 .
\end{equation*}

Consider now $w(x,t)$ in $Q_{r_0}$. For $t=0$, $w(x,0)\geq v(x,0)$, and for $r=r_0$, $w(x,t) \geq m+v\geq 0$. Then, according to Lemma~\ref{maximum-principle_weak}, the inequality $w(x,t) \geq 0$ holds everywhere in $\Omega$, and the result follows.

Observe that any particular point from $\mathbb{R}^{d+1}\cap\left(0<t\leq T\right)$ is contained in $Q_{r_0}$ for a sufficiently large $r_0$. Now, we have $w(x,t)=\frac{m}{r_0 ^2}(r^2+Kt)e^{\alpha t}+ v(x,t) \geq 0$. 
Taking the limit as $r_0 \to \infty$, $v(x,t) \geq 0$ $\forall(x,t)\in\Omega$ as desired.
\end{proof}

Similarly, one can prove:
\begin{lemma}\label{lemma-negative}
Suppose the function $u(x, t)$ is such that $ 0<a_0\le P(u) \leq \epsilon(r) r^2+ C$ with $\lim_{r \to \infty}\epsilon(r)=0$. 
Let $v\in \mathcal{C}^{2,1}$ and $L v \leq 0$,
then  $v(x,t)\leq \max v(x,0)$ 
for all $(x,t)\in\Omega $.
\end{lemma}

Now, from Theorem~\ref{u-to-v-thm} and Lemma~\ref{lemma-negative}, it follows:
\begin{corollary}
Assume $P(u)$ is bounded and $\tilde{L}u=0$, then $v=\int_0^u P(\xi) \,\rd\xi \leq \int_0^{\max u(x,0)}P(\xi)\,\rd\xi$.
\end{corollary}

However, we would like to make the bound on $v$ more precise. First we prove:

\begin{lemma} 
Let $w(x,t)= C F_{s,\beta}(x,t)$ for $t\geq 0$, where the constant $C>0$ and
\begin{equation}
\label{eq:barrier}
 F_{s,\beta} (x,t) = (t+1)^{-s} e^{-\frac{|x|^2}{4 \beta (t+1)}}
\end{equation}
is the barrier function. If  $0 < a_0\le P(u) \leq \beta$ and $s\leq\frac{ a_0d}{2\beta}$, then $Lw \geq 0$. 
\end{lemma}
\begin{proof}
First, consider the function $F_{s,\beta}(x,t)$ with constant $s >0$, for $t\gg1$. 
Second, compute
\begin{equation*}
    \begin{split}
        Lw &= C \Bigg( -\frac{s}{t+1}+ \frac{|x|^2}{4 \beta (t+1)^2}-P(u)\frac{|x|^2}{4 \beta ^2 (t+1)^2}+ P(u) \frac{d}{2 \beta (t+1)} \Bigg)(t+1)^{-s} e^{-\frac{|x|^2}{4 \beta (t+1)}}\\
        &= C \Bigg(\frac{|x|^2}{4 \beta ^2 (t+1)^2} \bigg(\beta- P(u)\bigg) + \frac{1}{2 \beta (t+1)}\bigg(d P(u)- 2s \beta \bigg)\Bigg) F_{s, \beta} (x,t) .
    \end{split}
\end{equation*}
Now, the conjecture of the lemma follows directly from this expression along with the assumptions.
\end{proof}

\begin{lemma}\label{land}
Let $w(x,t)= F_{s,\beta} (x,t)$ as in Eq.~\eqref{eq:barrier}. Assume that $s$ and  $\beta$ satisfy the same conditions as in the  previous lemma. If $Lv \leq 0$ and $w\geq v$ on $\Gamma$, then $v(x,t) \leq w(x,t).$
\end{lemma}

\begin{proof}
Consider,
\begin{equation*}
    \zeta(x,t) = w(x,t) - v(x,t) .
\end{equation*}
Then,
\begin{equation*}
  L\zeta =Lw - Lv  .
\end{equation*}
Since $Lv\leq 0$,
\begin{equation*}
    L\zeta \geq L w .
\end{equation*}

Then, due to the conditions on $s$ and $\beta$, $L\zeta \geq0$. 
Also, $w(x,0)\geq v(x,0)$ implies $\zeta(x,0)\geq 0$. 
Then by Lemma~\ref{maximum_principle_strong}, $\zeta(x,t) = F_{s,\beta} (x,t)- v(x, t) \geq 0$ in $\Omega$, which implies,
$$v(x,t) \leq  F_{s, \beta}(x,t)$$
in $\Omega$ for $s>0$ and a positive constant $\beta$.
\end{proof}

From the above results and the positivity of the function $u$, it  follows that  $0<a_0\leq P(u)\leq const $. Therefore, one can apply Aronson's estimate \citep{Aronson67} for the Green function of the divergence-form problem $\tilde{L}u=0$ with bounded coefficients. Namely: 
\begin{theorem}
  There exists  $G(x,t)$, Green function of the Cauchy problem for the equation $\tilde{L}u=0$, s.t.
  $$C_{2} t^{-\frac{d}{2}}e^{-\frac{|x|^2}{4 \beta_{-}t}} \leq G(x,t) \leq C_{1} t^{-\frac{d}{2}}e^{-\frac{|x|^2}{4 \beta_{+}t}}$$
  for some constants $C_1$ and $C_2$, at every $(x,t)\in\Omega$ ($t>0$).
\label{theorem:bound}
\end{theorem}

\begin{remark}
In Aronson's estimate \citep{Aronson67} discussed above, $\beta_{\mp}$ depend on the ellipticity constant and the spatial bounds of the divergence-form equation's coefficients. Specifically, if we write $\tilde{L}(\cdot)=\frac{\partial}{\partial t}(\cdot)-  \nabla\cdot(a(x,t) \nabla(\cdot))$, then the constants $\beta_\mp$ can be such that $\beta_{-} \leq a(x,t) \leq \beta_{+}$. Of course, in this representation needed to apply Aronson's estimate, the coefficient $a(x,t)$ depends on $u(x,t)$ directly. Then, since $u(x,t) \to 0$ as $t \to \infty$, it follows that $\beta_{-}=\beta_{+}=a_0$ (recall $a_0=P(0)$) in the limit as $t \to \infty$.
\end{remark}

Thus, finally, we have achieved our main conclusion: 
Let $u(x,t)$ be a solution of the Cauchy problem for $\tilde{L}u=0$, with $\tilde{L}$ as defined in Eq.~\eqref{conserv-dif_eq}, and  $ 0<a_0\le P(u) \leq \epsilon(r) r^2+ C$ with $\lim_{r \to \infty}\epsilon(r)=0$. Then, starting from a compact initial condition (such as that given by Eq.~\eqref{eq:u_ic_box} below, or other choices suitable for executing the proof with the barrier function from Eq.~\eqref{eq:barrier}), there exists constants $c_1$ and $c_2$ such that
\begin{equation}
    c_2 e^{-\frac{|x|^2}{4 a_0 t}} \leq t^{\frac{d}{2}}u(x,t) \leq c_1 e^{-\frac{|x|^2}{4 a_0 t}} \qquad\text{as}\quad t\to \infty.
\label{eq:final_bound}
\end{equation}
This mathematical result holds for any non-degenerate polynomial diffusivity function $P(u)$ as in Eq.~\eqref{eq:P(u)}, and thus significantly constrains the long-time asymptotic scalings that solutions to non-degenerate non-linear diffusion equations can exhibit. Determining, or at least constraining, the possible scaling behaviors of such solutions was our motivating scientific question, which we have now answered.

\section{Numerical experiments to illustrate the main mathematical results}
\label{sec:numerical}

In this section, we illustrate our mathematical results (in $d=1$ dimensions) with selected numerical simulations. Specifically, we show that the Green function $c_1 t^{-\frac{1}{2}} e^{-\frac{|x|^2}{4 a_0 t}}$, as in Eq.~\eqref{eq:final_bound}, does indeed describe, quantitatively, the long-time asymptotic behavior of the solutions of non-degenerate non-linear parabolic equations.

There are many numerical methods that one can use to solve the scalar parabolic equation~\eqref{p-u-equation} subject to the initial condition~\eqref{eq:u_ic} \citep{Strikwerda}. For simplicity, we use the \texttt{pdepe} subroutine of \textsc{Matlab} 2019b (Mathworks, Inc.), which is based on an auto-generated finite-element discretization and the method of lines, as described by \citet{Skeel90}. On a finite length domain, $x\in[-x_\mathrm{max},+x_\mathrm{max}]$ ($0< x_\mathrm{max}<\infty$), the asymptotic decay condition~\eqref{eq:decay_bc} must be replaced with an appropriate BC at $x=\pm x_\mathrm{max}$. In our numerical examples, we choose the interval to be large enough ($x_\mathrm{max}=200$ or larger, depending on the final simulation time $T$), so that this boundary condition does not influence the diffusion process of a localized initial condition. Then, we impose the ``natural'' (Neumann) boundary conditions
\begin{equation}
    \left.\frac{\partial u}{\partial x}\right|_{x=\pm x_\mathrm{max}} = 0 \qquad \forall t\in [0,T].
\label{eq:Neumann_bc_u}
\end{equation}
At least $10\,000$ $x$-grid points are used for the discretization, and time integration is performed by the adaptive, variable-order multistep stiff solver \texttt{ode15s} in \textsc{Matlab} \citep{Shampine97}.

\subsection{Compact initial condition}
\label{sec:num_comp_ic}

First, we take the initial condition, $u(x,0)=u_0(x)$, to be a box of unit area:
\begin{equation}
u_0(x) = \begin{cases} 
\displaystyle \frac{1}{2x_0}, &\quad|x|\le x_0,\\ 
0, &\quad|x|>x_0.
\end{cases}
\label{eq:u_ic_box}
\end{equation}
We take $x_0 =1$ without loss of generality, but we do note that various constants (in bounds, etc.) will depend on $x_0$. Then, we solve numerically the initial-boundary-value problem consisting of Eqs.~\eqref{p-u-equation}, \eqref{eq:Neumann_bc_u}, \eqref{eq:u_ic_box} on the finite space-time domain $[-x_\mathrm{max},+x_\mathrm{max}]\times(0,T]$.

\begin{figure}
    \centering
    \subfigure[]{\includegraphics[width=0.75\textwidth]{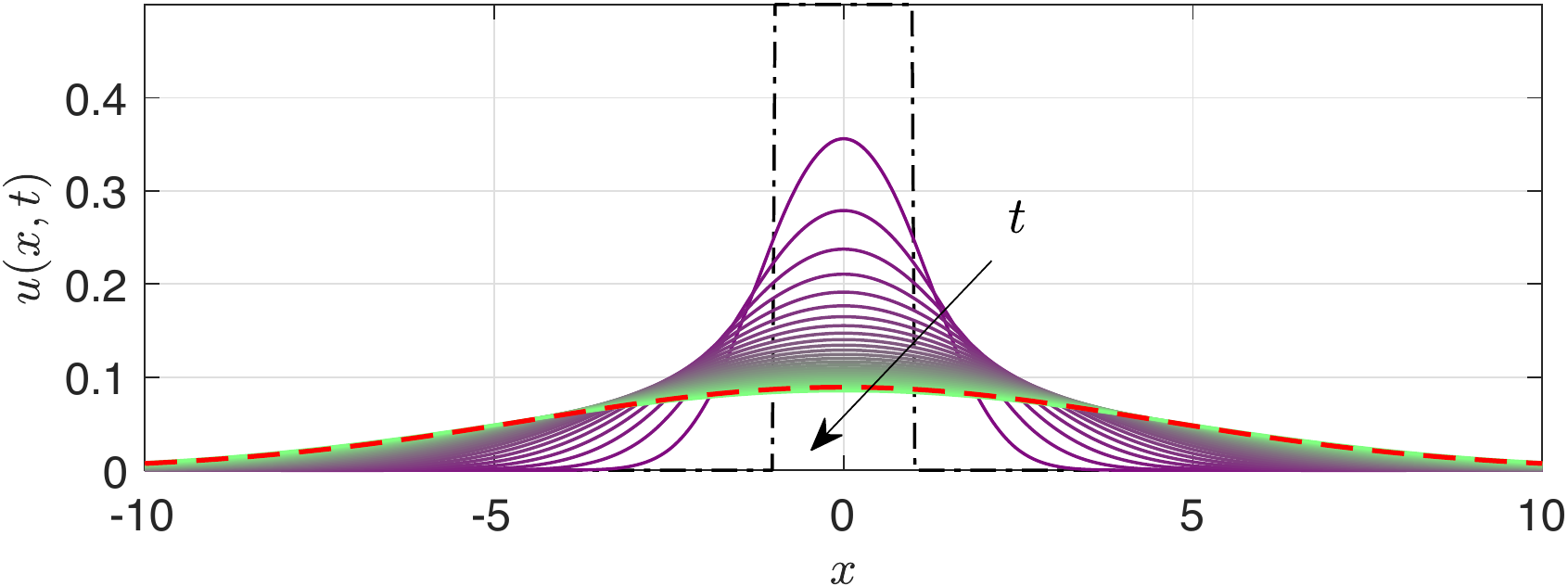}}
    \subfigure[]{\includegraphics[width=0.75\textwidth]{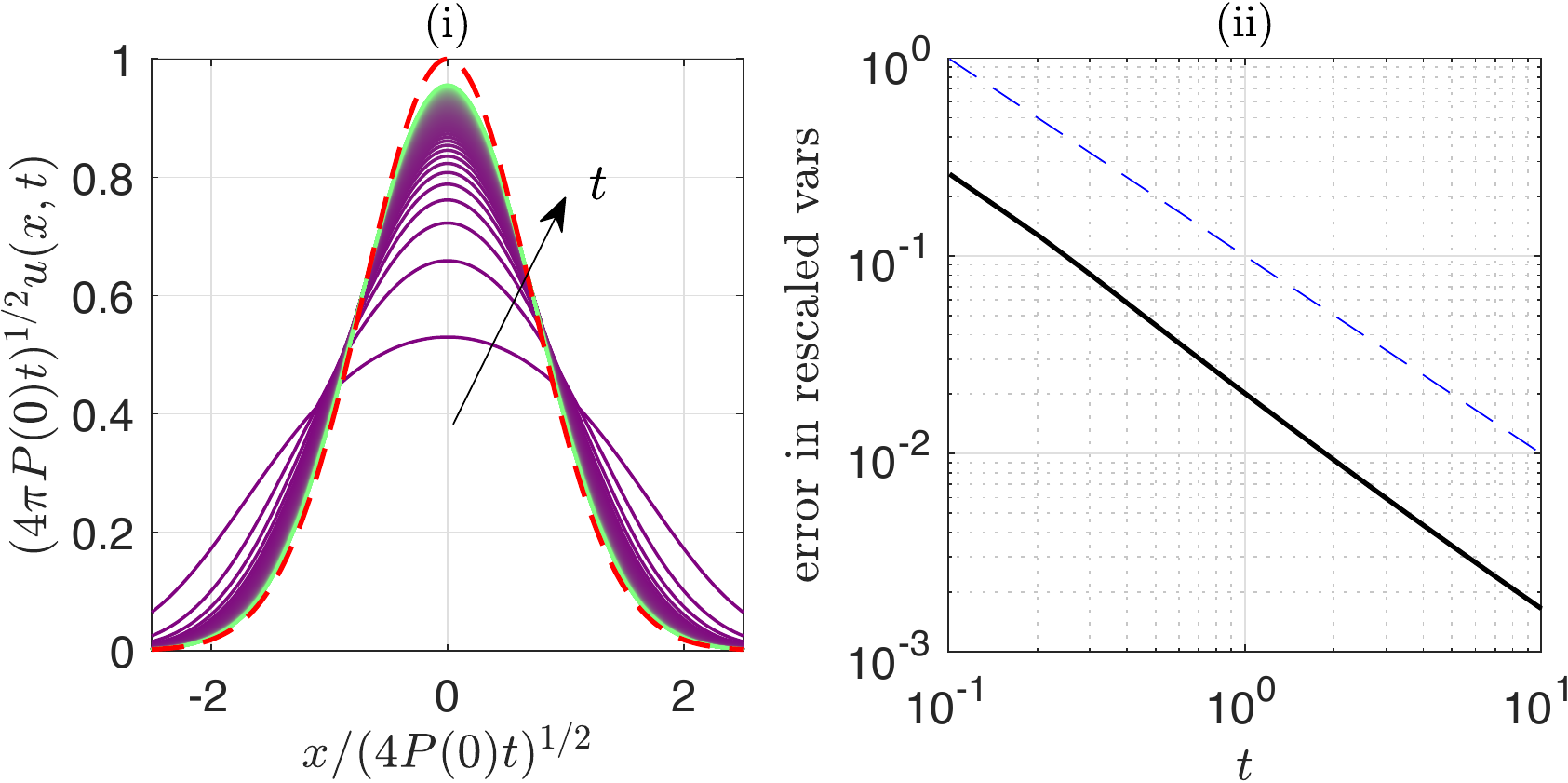}}
    \caption{(a) Time-evolution of $u(x,t)$ (visualized by 100 solution curves with color changing from purple/dark to green/light, from $t=0$ to $t=T=10$) starting from a unit box initial condition~\eqref{eq:u_ic_box} (dash-dotted curve) at $t=0$, with $P(u) = 1+u$. (b,i) At long times, when rescaled in the ``normal'' diffusive way, the solution appears to converge the Gaussian/fundamental solution~\eqref{eq:Gaussian_soln} of the linear problem with  $P(u) = 1$ (dashed curve). (b,ii) The $\mathbb{L}^2$ norm of the difference between the numerical solution of the non-linear problem and the Gaussian profile, having both been expressed in rescaled variables as in (b,i), decays algebraically in time (dashed line is a reference slope of $t^{-1}$). In all plots, the abscissas have been truncated for clarity.}
    \label{fig:example_diffusion}
\end{figure}

Figure~\ref{fig:example_diffusion}(a) shows an example time-evolution of this non-linear diffusion process. Clearly, the long-time numerical solution to the non-linear problem with $P(u) = 1 + u$ (note $P(0) = a_0 = 1$) converges, visually at least, as $t\to\infty$ to the fundamental (Gaussian) solution 
\begin{equation}
u_G(x,t) = \frac{1}{\sqrt{4\pi P(0)t}}e^{-\frac{x^2}{4P(0)t}}
\label{eq:Gaussian_soln}
\end{equation}
of the linear problem with $P(u) = 1$ and an initial condition of unit area, $\int_{-\infty}^{+\infty}u_0(x)\,\rd x=1$. Note that, since we are interested in the long-time asymptotics (specifically, the \emph{scaling} of the solution), we neglect details arising from the fact the initial condition~\eqref{eq:u_ic_box} \citep[see][]{Kleinstein71,Witelski98,Christov12} is not a point source (Dirac $\delta$) for which, specifically, the exact solution to the linear problem is the fundamental solution in Eq.~\eqref{eq:Gaussian_soln}.

On rescaling the solution $u(x,t)$ of the non-linear diffusion problem using the ``normal'' (linear) diffusion scalings as $u(x,t)\mapsto(4\pi P(0)t)^{1/2}U(\xi)$ with $\xi=x/(4P(0)t)^{1/2}$, in Fig.~\ref{fig:example_diffusion}(b,i), we observe the curves begin to approach the Gaussian profile $U(\xi) = e^{-\xi^2}$. The bound, Eq.~\eqref{eq:final_bound} proved in \S\ref{sec:max_ubd}, is difficult to evaluate point-wise numerically because $\beta$ in the barrier function~\eqref{eq:barrier} (and $\beta_\mp$ in Theorem \ref{theorem:bound}) can depend, in particular, on $t$. However, in the rescaled coordinates, it is possible to evaluate the bound in an $\mathbb{L}^2$ sense by computing the norm of the difference between the rescaled solution of the non-linear problem and the Gaussian profile $U(\xi) = e^{-\xi^2}$. This ``error'' decays algebraically in time, as Fig.~\ref{fig:example_diffusion}(b,ii) shows. Note that the algebraic decay in $t$ is expected from previous estimates of the convergence rate of solutions to nonlinear parabolic equations, starting from arbitrary initial data, towards their self-similar intermediate asymptotics  \citep{Kleinstein71,Witelski98,Bernoff10}. The numerical observation that the norm of the ``error'' decays in time indicates that the bound proved in \S\ref{sec:max_ubd} is accurate, and the solution of the non-linear problem converges (as $t\to\infty$) to the Gaussian profile.

\begin{figure}
    \centering
    \subfigure[]{\includegraphics[width=0.75\textwidth]{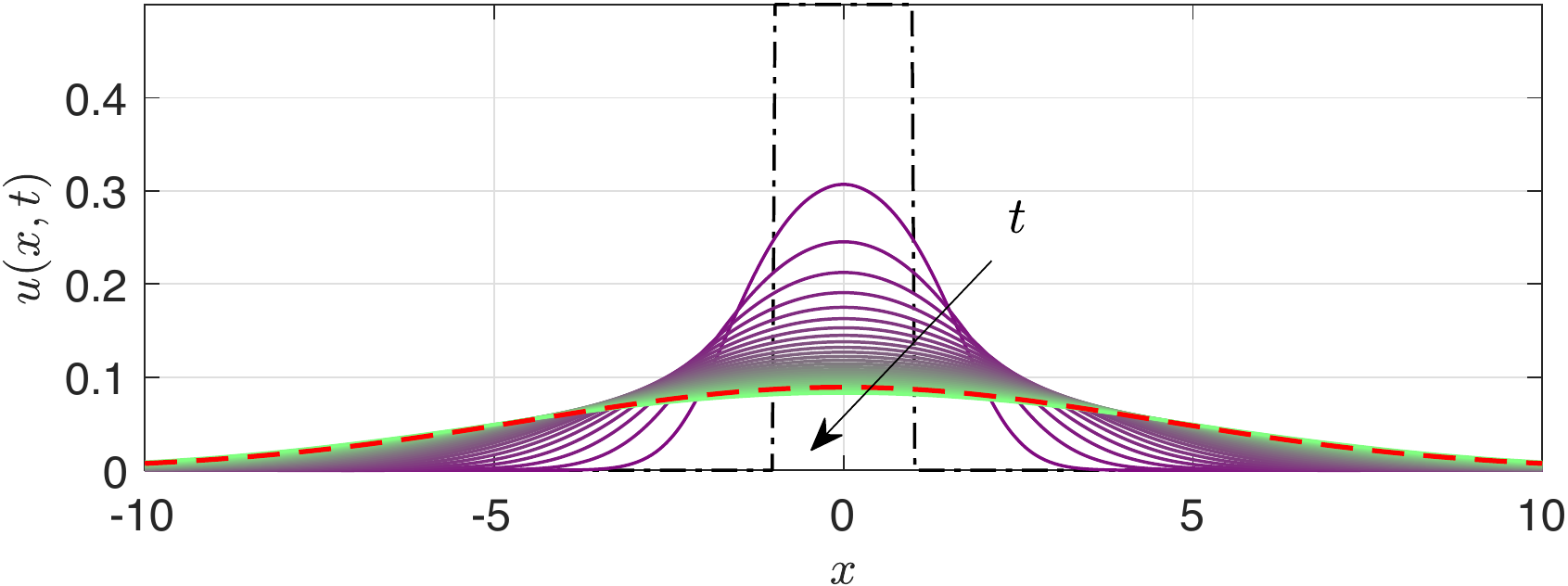}}
    \subfigure[]{\includegraphics[width=0.75\textwidth]{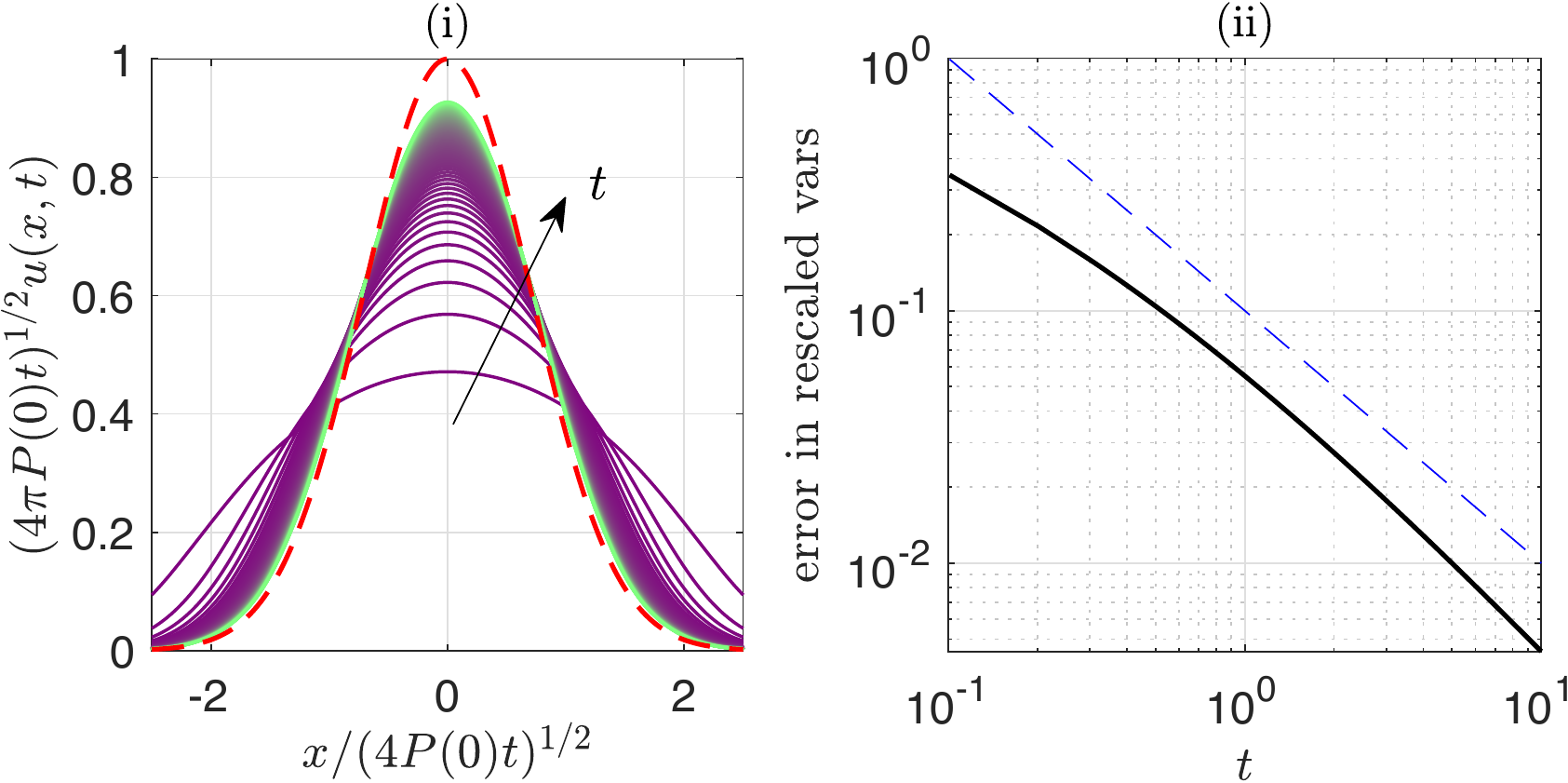}}
    \caption{Same as Fig.~\ref{fig:example_diffusion} but with the diffusivity $P(u) = 1 + u + 10u^2$ being a quadratic polynomial.}
    \label{fig:example2_diffusion}
\end{figure}

Figure~\ref{fig:example2_diffusion} shows the equivalent of Fig.~\ref{fig:example_diffusion} but with $P(u) = 1 + u + 10u^2$. Clearly, as predicted by the mathematical theory, the long-time asymptotics are similar but the constants in the bound~\eqref{eq:final_bound}, and prior theorems and lemmas, change. Furthermore, due to $10u^2$ dominating $1+u$ at early times, the approach to the ultimate long-time asymptotics takes longer, thus the convergence in Fig.~\ref{fig:example2_diffusion}(b) looks worse than in Fig.~\ref{fig:example_diffusion}(b) for the same integration time-interval: $t\in(0,T=10]$.

\begin{figure}
    \centering
    \subfigure[]{\includegraphics[width=0.75\textwidth]{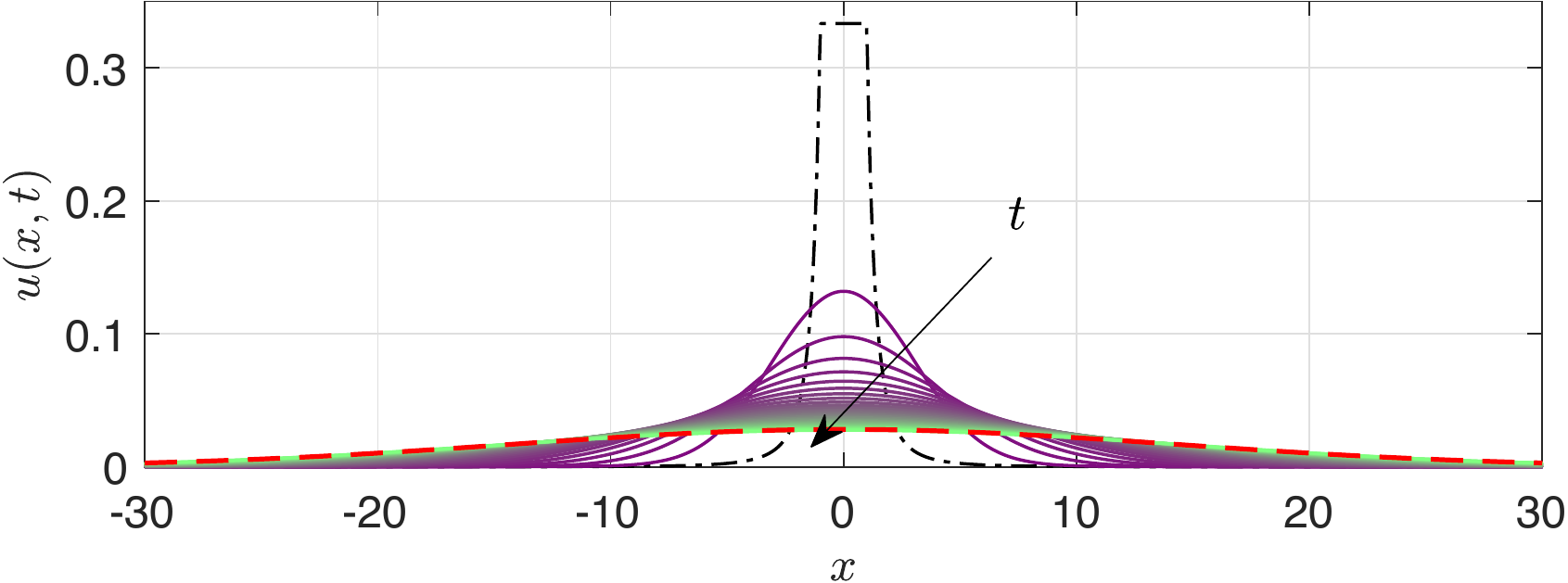}}
    \subfigure[]{\includegraphics[width=0.75\textwidth]{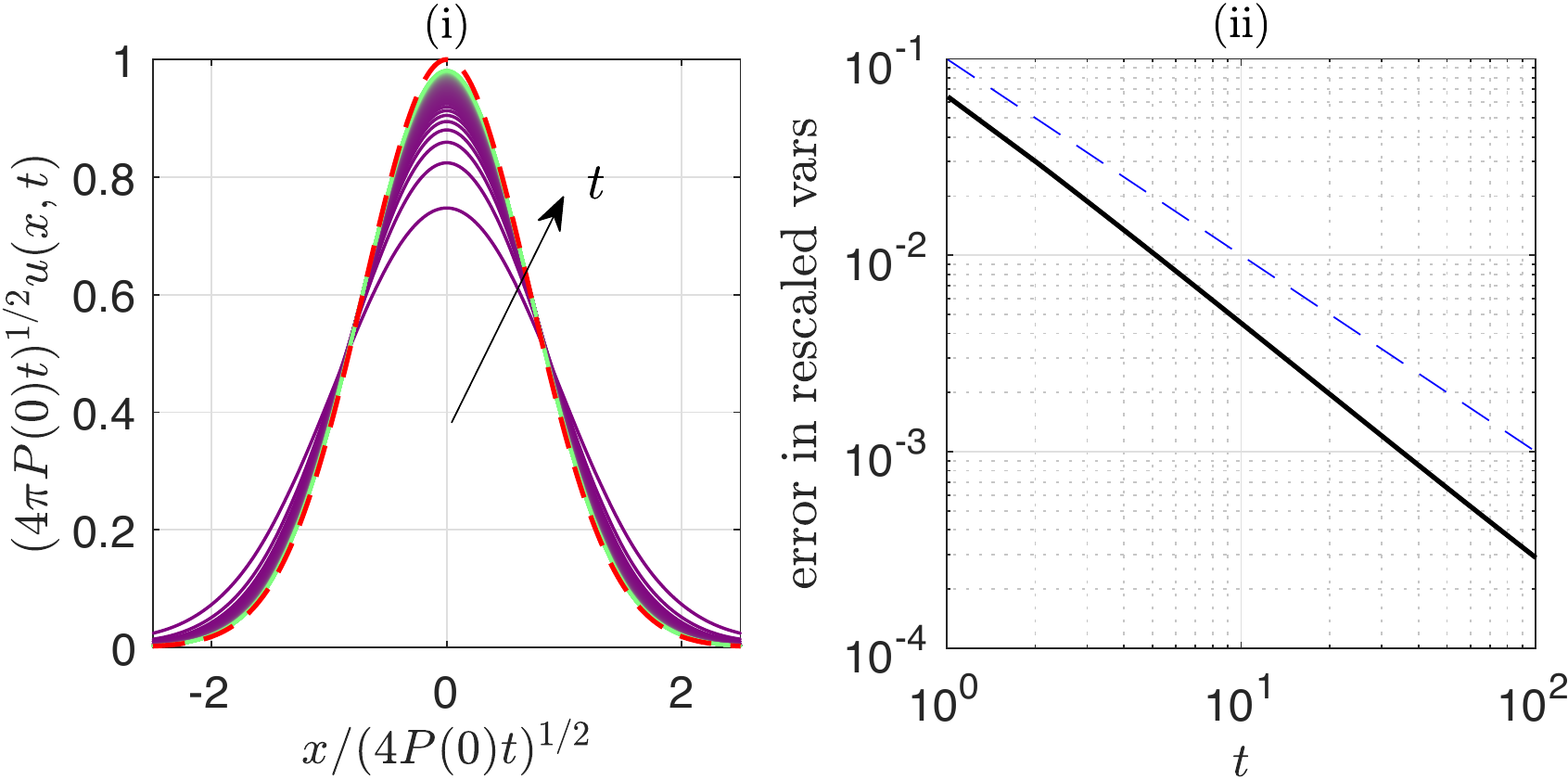}}
    \caption{Same as Fig.~\ref{fig:example_diffusion} but starting from the non-compact initial condition in Eq.~\eqref{eq:u_ic_decay}.}
    \label{fig:example3_diffusion}
\end{figure}

\subsection{Non-compact initial condition with algebraic decay}
\label{sec:num_noncomp_ic}

Second, we take $u_0(x)$ to have slow algebraic decay at infinity as per Eq.~\eqref{eq:decay_bc}:
\begin{equation}
u_0(x) = \left(\frac{\gamma-1}{2[x_0(\gamma-1)+1]}\right) \begin{cases} 
\displaystyle 1, &\quad|x|\le x_0,\\ 
(|x| - x_0 + 1)^{-\gamma}, &\quad|x|>x_0.
\end{cases}
\label{eq:u_ic_decay}
\end{equation}
Note that the function was chosen so that we still have unit ``mass,'' i.e., $\int_{-\infty}^{+\infty} u_0(x) \, \rd x = 1 < \infty$, which restricts the decay rate to  $\gamma > 1$. Again, $x_0 = 1$ without loss of generality.

Let us take $\gamma=3$ for the remainder of this numerical experiment. Figure~\ref{fig:example3_diffusion} shows the equivalent of Fig.~\ref{fig:example_diffusion} but starting from the initial condition in Eq.~\eqref{eq:u_ic_decay}. Clearly, the long-time asymptotics are similar as in the previous two examples in \S\ref{sec:num_comp_ic} starting from compact initial conditions. Of course, the constants in the various theorems and lemmas change, and there are quantitative differences in the convergence process. Importantly, to reach ``sufficiently large'' $t\gg1$ requires a much longer total simulation time $T$. Here, we had to use $x_\mathrm{max}=2\,000$, $100\,000$ $x$-grid points and integrate up to $T=100$ to observe the long-time asymptotics. Values of $\gamma$ closer to $1$ require even more computational care and resources.

This example is interesting as our construction, leading up to the main result in \S\ref{sec:max_ubd}, was based on a barrier function, Eq.~\eqref{eq:barrier}, with exponential decay as $|x|\to\infty$. However, the initial condition (and, hence, the solution) in this example does not satisfy this condition, having only algebraic decay. Yet, the numerical results suggest that the final bound obtained in Eq.~\eqref{eq:final_bound}, holds just as well for this example, a case that violates the assumptions of the theorems and lemmas proved above.

%%%%%%%%%%%%%%%%%%%%%%%%%%%%%%%%%%%%%%%%%%%%%%%%%%%%

\section{Conclusion}
\label{sec:conclusion}

In this paper, we discussed how Einstein's random walk paradigm can be employed to derive non-linear parabolic equations in non-divergence form, as a continuum description of a random-walk diffusion process. Then, by proving that a mapping between divergence-form and non-divergence-form parabolic equations exists, we connected the derivation from the Einstein paradigm to the traditional derivation of the diffusion equation from a conservation law, i.e., the continuity equation along with Fick's (or Fourier's, etc.) law. 

The mapping theorem enabled us to obtain accurate estimates for the Green function of the Cauchy problem for the divergence-form non-linear diffusion equation. These estimates were used to prove bounds on the solution of the non-linear diffusion equation, from both above and below, and establish the long-time asymptotics of the non-linear equation's solutions. Specifically, we proved that the solution to the non-degenerate non-linear problem converges to the fundamental solution (Gaussian distribution) of the linear diffusion problem, leading to the bound in Eq.~\eqref{eq:final_bound}, which is valid from above and from below.  Numerical simulations for some example non-linear equations quantitatively support the mathematical results proved via Aronson's estimate.

Importantly, the present work sheds light on the issue of anomalous diffusion scalings in certain areas of applied physics. Specifically, by proving that the fundamental solution of the linear diffusion problem is the long-time asymptotic behavior (from below and from above) of solutions to \emph{any} non-degenerate parabolic equation with strictly positive polynomial diffusivity, from arbitrary initial data, we are led to suggest that any ``anomaly'' in the scalings might be an artefact of the short time-duration of an experiment. Ultimately, the physics must justify whether degeneracy of the governing equation is expected (or not) and why. Indeed, for degenerate non-linear parabolic equations, a wealth of different scaling functions and transformations (different from the one defined in Eq.~\eqref{eq:barrier}), leading to bounds different from Eq.~\eqref{eq:final_bound}, are allowed \citep[see, e.g.,][]{Barenblatt96,Witelski98}. Non-degeneracy, on the other hand as we proved above, imposes \emph{strict} restrictions on the asymptotic scaling behavior of a non-linear diffusion equation's solutions.

Interestingly, the numerical example in \S\ref{sec:num_noncomp_ic} suggests that our mathematical results also hold for initial conditions and solutions that lack the exponential decay required by the barrier function in Eq.~\eqref{eq:barrier}. Therefore, in future work, it would be of interest to attempt, or to determine whether it is even possible, to generalize the proof in \S\ref{sec:max_ubd} using a barrier function with algebraic decay as $|x|\to\infty$.

Finally, it should be noted that, in an interesting paper, \citet{Bricmont94} used the re-normalization group (RG) method \citep{Wilson83} to prove certain results about the asymptotic behaviour of non-linear diffusion equations. Specifically, taking the difference between the fundamental solution of the linear problem (Gaussian distribution) and the solution of the non-linear PDE in divergence form, they are able to prove convergence at a point $(x,t)=(\sqrt{t},t)$. The RG method is generic, and it does not use the maximum principle. On the other hand, our approach is based on maximum principle leading to an Aronson estimate. Importantly, our approach is applicable to both non-linear equations in both divergence and in non-divergence form; in the latter case, the coefficients may even depend on the spatial variable. As a result, our approach leads to an accurate estimate, from both above and below, of the solution of the non-linear equation in terms of Aronson's Green function. That is to say, the estimates proved in this work cannot be obtained via the RG method of \citet{Bricmont94}.

\section*{Acknowledgements}
Acknowledgment is made to the donors of the American Chemical Society Petroleum Research Fund for partial support of I.C.C.,  under ACS PRF award \# 57371-DNI9, during the course of this research. I.C.C.\ thanks H.A.\ Stone for many insightful discussion about diffusion, scalings, and self-similarity.

\section*{Data Availability Statement}
The data that support the findings of this study are available from the corresponding author upon reasonable request.

%%%%%%%%%%%%%%%%%%%%%%%%%%%%%%%%%%%%%%%%%%%%%

\bibliography{references.bib}

\end{document}